\documentclass[oneside,english]{amsart}
\usepackage[T1]{fontenc}
\usepackage[latin9]{inputenc}
\usepackage{color}
\usepackage{float}
\usepackage{mathtools}
\usepackage{amstext}
\usepackage{amsthm}
\usepackage{amssymb}
\usepackage{graphicx}
\PassOptionsToPackage{normalem}{ulem}
\usepackage{ulem}

\makeatletter
\numberwithin{equation}{section}
\theoremstyle{plain}
\newtheorem{thm}{\protect\theoremname}[section]
\theoremstyle{definition}
\newtheorem{condition}[thm]{\protect\conditionname}
\theoremstyle{definition}
\newtheorem{defn}[thm]{\protect\definitionname}
\theoremstyle{remark}
\newtheorem{rem}[thm]{\protect\remarkname}
\theoremstyle{plain}
\newtheorem{prop}[thm]{\protect\propositionname}
\theoremstyle{plain}

\theoremstyle{definition}
\newtheorem{example}[thm]{\protect\examplename}
\theoremstyle{remark}
\newtheorem*{acknowledgement*}{\protect\acknowledgementname}
\newtheorem{cor}{Corollary}[section] 

\usepackage{amsthm}

\newtheorem*{proposition*}{Proposition}

\makeatother

\usepackage{babel}
\providecommand{\acknowledgementname}{Acknowledgement}
\providecommand{\conditionname}{Condition}
\providecommand{\definitionname}{Definition}
\providecommand{\examplename}{Example}
\providecommand{\lemmaname}{Lemma}
\providecommand{\propositionname}{Proposition}
\providecommand{\remarkname}{Remark}
\providecommand{\theoremname}{Theorem}
\numberwithin{equation}{section}
\theoremstyle{plain}
\theoremstyle{remark}
\newtheorem*{rem*}{\protect\remarkname}
\theoremstyle{plain}

\theoremstyle{plain}
\newtheorem*{prop*}{\protect\propositionname}
\theoremstyle{definition}
\theoremstyle{plain}
\newtheorem*{thm*}{\protect\theoremname}
\theoremstyle{remark}

\begin{document}
\title[M\"{o}bius-Type Structures in Singular Semi-Riemannian Manifolds]{M\"{o}bius-Type Structures in Non-Orientable Singular Semi-Riemannian
Manifolds}
\author{Nathalie E. Rieger}
\address{Department of Mathematics, University of Zurich, Winterthurerstrasse
190, 8057 Zurich, Switzerland}
\curraddr{Mathematics Department, Yale University, 219 Prospect Street, New
Haven, CT 06520, USA}
\email{n.rieger@yale.edu}

\begin{abstract}
Our objective is to illuminate the global structure of non-orientable
manifolds with signature-changing metrics, with particular emphasis on global topological obstructions. Using explicit geometric constructions
based on the topology of the M\"{o}bius strip, we produce examples
of crosscap manifolds where the gluing junction coincides with the locus
of signature change. 
Our main result shows that on non-orientable compact surfaces, the radical of such metrics cannot be everywhere transverse along the hypersurface of signature change. In particular, metrics arising from the transformation prescription $\tilde{g}=g+fV^{\flat}\otimes V^{\flat}$, with $g$ a Lorentzian metric and $f$ a smooth interpolation function, necessarily fail to satisfy the transversality condition. This obstruction is of purely global origin and is closely related to topological invariants such as the Euler characteristic and the non-existence of nowhere-vanishing vector fields. These results demonstrate that non-orientability imposes intrinsic limitations on the class of admissible signature-type changing metrics.
\end{abstract}

\keywords{M\"{o}bius strip, causality, singular semi-Riemannian geometry, Lorentzian geometry, signature change, time orientability, singular metric, quantum cosmology}

\maketitle
\address{Department of Mathematics, Yale University, USA}

\email{n.rieger@yale.edu}

\section{Introduction}

We explore manifolds with signature change, in which a Riemannian region transitions to a Lorentzian region. 
At the locus of signature change, exactly one eigenvalue of the metric tensor vanishes: continuity of the eigenvalue spectrum requires one eigenvalue to pass through zero across the transition. Consequently, the metric becomes degenerate on this hypersurface, and its inverse no longer exists. Thus, any continuous signature-changing metric is necessarily degenerate there~\cite{Dray - Gravity and Signature Change}. Such metrics are referred to as singular metrics.

\textbf{~}\\ Formally, we study a smooth, symmetric $(0,2)$-tensor field $\tilde{g}$
on a smooth manifold $M$ that becomes degenerate on a smoothly embedded
hypersurface $\mathcal{H}\subset M$. The bilinear type of $\tilde{g}$
changes upon crossing $\mathcal{H}$. Of particular interest are the
cases where the radical of $\tilde{g}$ at points of $\mathcal{H}$
is transverse to the hypersurface.

\textbf{~}\\ This framework is of significant physical relevance, particularly in theories of quantum gravity and quantum cosmology. Signature change is a core feature in models relying on Wick rotation (e.g., Euclidean quantum gravity and the no-boundary proposal) and is used to calculate instanton transition rates between different phases of the early universe. Our mathematical study of the geometric constraints imposed by non-orientable topologies is thus motivated by the need for rigorous foundation and analysis of novel phenomena in these physical contexts.

~

This article builds on the framework developed in~\cite{Hasse + Rieger-Transformation, Hasse + Rieger-Local Transformation, Hasse + Rieger-Loops},
in particular the Transformation Theorem, which asserts that (locally)
the metric $\tilde{g}$ of a transverse signature-changing manifold
$(M,\tilde{g})$ is equivalent to a metric obtained from some Lorentzian
metric $g$:

\begin{thm*}[Global Transformation Theorem, transverse radical]
\label{Transformation-Theorem-(global)-1}
Let $M$ be a transverse, signature-type changing manifold of $\dim(M)=n\geq2$,
which admits in $M_{R}\cup\mathcal{H}$ a smoothly defined non-vanishing
line element field that is transverse to the boundary $\mathcal{H}$.
Then the metric $\tilde{g}$ associated with a signature-type changing
manifold $(M,\tilde{g})$ is a transverse, type-changing metric with
a transverse radical if and only if $\tilde{g}$ is obtained from
a Lorentzian metric $g$ via the Transformation Prescription $\tilde{g}=g+fV^{\flat}\otimes V^{\flat}$,
where, for all $q\in \mathcal H \coloneqq f^{-1}(1)=\{p\in M\colon f(p)=1\}$,
\[
df(q)\neq0 \quad \text{and} \quad (df(V))(q)\neq0.
\]
\end{thm*}

We also rely on the Transformation Prescription, which
provides a procedure to transform an arbitrary Lorentzian manifold
into a singular signature-changing manifold. 

\begin{prop*}[Transformation Prescription]
\label{Proposition Transformation-Prescription-1}
Let $(M,g)$ be a (not necessarily time-orientable) Lorentzian manifold
of $\dim n\geq2$. Then we obtain a signature-type changing metric
$\tilde{g}$ via the Transformation Prescription $\tilde{g}=g+fV^{\flat}\otimes V^{\flat}$,
where $f\colon M\longrightarrow\mathbb{R}$ is a smooth transformation
function and $V$ is one of the unordered pair $\{V,-V\}$
of a global smooth non-vanishing line element field.
\end{prop*}

We assume the reader's
familiarity with these results. Our aim here is to extend the study
of manifolds with signature change, with particular emphasis on constructions
that exhibit non-orientable global structures. To illuminate the global structure of non-orientable signature-changing
manifolds, we construct explicit examples modeled on the topology
of the M\"{o}bius strip. By gluing diffeomorphic, connected nonempty
boundary components of smooth manifolds with collar neighborhoods,
we obtain new manifolds whose topology resembles that of a crosscap,
with the gluing locus serving as the hypersurface of signature change.

\textbf{~}\\ Some of our examples lie outside the class introduced in~\cite{Hasse + Rieger-Transformation, Hasse + Rieger-Local Transformation}, as they give rise to metrics whose radical is of mixed character, i.e., the radical of the metric transitions from transverse to tangent at $\mathcal{H}$. We then apply the Transformation Theorem
to the M\"{o}bius strip, transforming it from a Lorentzian manifold
into a signature-type changing manifold with a transverse radical. 

~

Finally, we investigate whether the resulting metrics can be written
in the form 
\[
\tilde{g}=g+f(V^{\flat}\otimes V^{\flat}),
\]
where $f$ is a smooth transformation function that interpolates between
the Riemannian and Lorentzian regions, $g$ is some Lorentzian metric,
and $V^{\flat}$ denotes the metric dual of a vector field $V$. Since $V$ is a line element field that remains non-zero at every point 
in $M$, it can be normalized such that $g_{\mu\nu}V^{\mu}V^{\nu}=-1$.\footnote{Here ``timelike'' refers to the Lorentzian metric $g$.}
The hypersurface of signature change then appears as $\mathcal{H}=f^{-1}(1)$.
This perspective highlights the diversity of possible structures for
non-orientable manifolds with signature change.

\medskip
\noindent\textbf{Main result (informal).}
The principal result of this paper is a global obstruction to the existence
of signature-type changing metrics with a transverse radical on non-orientable
compact manifolds. In particular, we show that on manifolds with the topology
of the crosscap, any metric obtained via the transformation prescription
\[
\tilde{g} = g + f(V^\flat \otimes V^\flat)
\]
necessarily fails to have a radical that is everywhere transverse along the
hypersurface of signature change.
\medskip

This shows that the existence of transverse type-changing metrics is not
purely a local phenomenon, but is fundamentally constrained by global
topological properties such as non-orientability and Euler characteristic.

\subsection{Preliminaries and fundamental definitions}

Let $(M,\tilde{g})$ be a smooth manifold of dimension $\dim M\geq2$.
Following~\cite{Kossowksi + Kriele - Signature type change and absolute time in general relativity},
we call $(M,\tilde{g})$ a \textit{transverse type-changing singular
semi-Riemannian manifold} if it satisfies the following condition.

\begin{condition}
The metric $\tilde{g}$ is a smooth, symmetric, $(0,2)$-tensor
field on $M$ which is degenerate on $\mathcal{H}$. Moreover, $\tilde{g}$ is a codimension-$1$\textbf{
}\textit{transverse type-changing metric}, meaning that 
\[
d(\det(\tilde{g}_{\mu\nu}))_{q}\neq0\text{ for all } q\in\mathcal{H},
\]
and for any local coordinate system $\xi=(x^{0},\ldots,x^{n-1})$
around $q$. As a consequence, the degeneracy locus $$\{q\in M\!\!:\tilde{g}\!\!\mid_{q} is\;degenerate\}=:\mathcal{H}\subset M$$
is a smoothly embedded hypersurface in $M$. The bilinear type of
$\tilde{g}$ changes when crossing $\mathcal{H}$.
\end{condition}

At each point $q\in\mathcal{H}$, there exists a one-dimensional radical
$\textrm{Rad}_{q}\subset T_{q}M$, defined as the subspace 
\[
\textrm{Rad}_{q}:=\{w\in T_{q}M\mid\tilde{g}_{q}(w,v)=0\text{ for all }v\in T_{q}M\}.
\]
The radical may be either tangent or transverse to $\mathcal{H}$.

\begin{defn}[Character of the radical]
We say that the radical of a metric is transverse to $\mathcal{H}$ if, at every point of 
$\mathcal{H}$, the radical distribution is not contained in $T\mathcal{H}$. If it is contained in 
$T\mathcal{H}$ at every point, we call it purely tangent. In all remaining cases---when it is transverse at some points and tangent at others---we say that the radical has mixed character; equivalently, it is non-transverse with respect to $\mathcal{H}$.
\end{defn}

To focus on the transverse case, we impose the following condition.

\begin{condition}\label{transverse radical}
The radical $\textrm{Rad}_{q}$ is transverse for all $q\in\mathcal{H}$,
i.e. 
\[
\textrm{Rad}_{q} \oplus T_{q}\mathcal{H}=T_{q}M\text{ for all }q\in\mathcal{H}.
\]
Equivalently, $\textrm{Rad}_{q}$ is not tangent to $\mathcal{H}$
at any point $q$, i.e., $\textrm{Rad}_{q}\not\subset T_{q}\mathcal{H}$.
\end{condition}

We also reiterate the definition for pseudo-timelike curves~\cite{Hasse + Rieger-Loops}.

\begin{defn}[Pseudo-timelike curve]
\textbf{\label{Definition Pseudo-timelike curve}}Let $M=M_{L}\cup\mathcal{H}\cup M_{R}$ be an $n$-dimensional
transverse type-changing singular semi-Riemannian manifold, $g$ be
a symmetric type-changing metric, and $\mathcal{H}$ the locus of
signature change. Let $\gamma\colon[a,b]\to M$ be a continuous and
differentiable curve, with $[a,b]\subset\mathbb{R}$ and $-\infty<a<b<\infty$.
We call $\gamma=\gamma^{\mu}(u)=x^{\mu}(u)$ in $M$ a pseudo-timelike
(respectively, pseudo-spacelike) curve if 

~\linebreak{}
(i) $\textrm{Im}(\gamma)\cap M_{L}\neq\varnothing$, i.e. $\gamma$
has image points in the Lorentzian region; and 

~\linebreak{}
(ii) for every generalized affine parametrization of $\gamma$ in
$M_{L}$ there exists $\varepsilon>0$ such that $g(\gamma',\gamma')<-\varepsilon$
(respectively, $g(\gamma',\gamma')>\varepsilon$).
\end{defn}

\subsubsection{Orientability and pseudo-orientability}

Before proceeding, we recall several orientability notions relevant
to signature-changing manifolds.

\begin{defn}
~\cite{Bott + Tu} A smooth $n$-dimensional manifold $M$ is\textit{
orientable} if and only if it admits a smooth global nowhere vanishing
$n$-form (also called a top-ranked form).
\end{defn}

We also recall the following notions from~\cite{Hasse + Rieger-Loops},
where $(M,\tilde{g})$ denotes a smooth signature-changing manifold
(possibly with boundary). We assume that one connected
component of $M\setminus\mathcal{H}$ is Riemannian, denoted by $M_{R}$,
while all other connected components $(M_{L_{\alpha}})_{\alpha\in I}\subseteq M_{L}\subset M$
are Lorentzian, where 
\[
M_{L}:=\underset{\alpha\in I}{\bigcup}M_{L_{\alpha}}
\]
represents the Lorentzian region. Furthermore, we assume throughout that the
point set $\mathcal{H}$, where $\tilde{g}$ becomes degenerate, is not empty.

\begin{defn}[Pseudo-timelike] \label{def:Pseudo-timelike}A vector field
$V$ on $(M,\tilde{g})$ is \textit{pseudo-timelike} if and only if
$V$ is timelike in $M_{L}$ and its integral curves are pseudo-timelike.
\end{defn}

\begin{defn} [Pseudo-time orientable] \label{def:Pseudo-time-orientable}
A signature-type changing manifold $(M,\tilde{g})$ is \textit{pseudo-time
orientable} \textcolor{black}{if and only if} the Lorentzian region
$M_{L}$ is time orientable.
\end{defn}

\begin{defn}[Pseudo-space orientable] \label{def:Pseudo-space-orientable}
A signature-type changing manifold $(M,\tilde{g})$ of dimension $n$
is \textit{pseudo-space orientable} if and only if it admits a continuous
non-vanishing spacelike $(n-1)$-frame field on $M_{L}$. This is
a set of $n-1$ pointwise orthonormal spacelike vector fields on $M_{L}$.
\end{defn}

\subsection{Summary of main results}

The first result concerns the causal structure of the rotating Minkowski model
introduced in Subsection~\ref{subsec:Rotating-Minkowski-Metric}. We show that, with respect to the chosen global
time orientation, the null boundaries of the stationary stripes act as one-way
causal barriers: future-directed causal curves entering the even stripes \(M_{2k}\)
cannot leave them, whereas the odd stripes \(M_{2k-1}\) cannot be entered, a property that plays a crucial role in understanding the causal
behavior of the crosscap metric studied later. 

\begin{prop}[One-way causal barriers]
\label{prop:trapped curves}
Let $(M,g)$ be the two-dimensional manifold $M=\mathbb{R}^2$ with coordinates $(t,x)$
equipped with the ``rotating Minkowski'' metric
\[
g = -\cos(2\varphi)\,(dt)^2
    + 2\sin(2\varphi)\,dt\,dx
    + \cos(2\varphi)\,(dx)^2,
\qquad \varphi=\pi x.
\]
For each $k\in\mathbb{Z}$ define the stationary stripe
\[
M_k:=\left\{(t,x)\in M:
k-\frac14 < x < k+\frac14
\right\}.
\]
Then $\partial_t$ is a Killing field and is timelike on each $M_k$.

Let the time orientation be induced by the global timelike vector field
\[
V=(\cos\varphi)\,\partial_t-(\sin\varphi)\,\partial_x .
\]
With respect to this time orientation, the null boundary components of the
stationary stripes are one-way causal barriers. More precisely, every
future-directed causal curve which enters an even stationary stripe $M_{2k}$
cannot leave it, whereas no future-directed causal curve can enter an odd
stationary stripe $M_{2k-1}$.
\end{prop}

The second result clarifies the orientability properties of manifolds that
undergo a change of signature. Even when both pseudo-time and pseudo-space
orientability hold, orientability in the usual sense may still fail.

\begin{prop}[Orientability] 
\label{prop:orientable+timeorientable-not_orientable}
Even if a transverse, signature-type changing manifold $(M,g)$ with
a transverse radical is pseudo-time orientable and pseudo-space orientable,
it is not necessarily orientable.
\end{prop}

Finally, in Subsection~\ref{subsec:Compact Moebius-Strip} we analyze in detail the explicit crosscap model
constructed via the adjunction space description in
Equation~\eqref{eq: adjunction space}. The following proposition summarizes the
main structural properties of this metric, including the behavior of the
radical and the precise manner in which signature change occurs. 

\begin{prop}[Radical and signature change on the crosscap]
\label{prop:crosscap-proposition}
Let $(C,g)$ be the crosscap (see Eq.~\ref{eq: adjunction space}) with metric  \[ g = (1-t^{2})\,dt^{2} + 2tx\,dt\,dx + (1-x^{2})\,dx^{2}, \] and define the degeneracy locus \[ \mathcal{H} = \{ (t,x) \in C : t^{2}+x^{2}=1 \}. \]
Then: \begin{enumerate} \item $(C,g)$ is a \emph{signature-changing manifold}, with Lorentzian signature in the region $t^{2}+x^{2} > 1$ and Riemannian signature in $t^{2}+x^{2} < 1$. \item The \emph{radical} at each point $q \in \mathcal{H}$ is 1-dimensional: \[ \mathrm{Rad}_{q} = \mathrm{span} \Bigl\{ (1, \sqrt{1-t^{2}}/t)^{T} \Bigr\} = \mathrm{span} \Bigl\{ (\sqrt{1-x^{2}}/x,1)^{T} \Bigr\}, \] where $T$ denotes the transpose of the matrix with respect to the chosen coordinates. \item The radical is of mixed character with respect to $\mathcal{H}$ ($\mathrm{Rad}_{q}$ is tangent iff $t^2 = x^2 = \frac{1}{2}$), with the radical being tangent at a measure-zero set. \end{enumerate}
Consequently, $(C,g)$ provides a compact example of a signature-changing manifold that \emph{cannot} be obtained from a Lorentzian manifold via the Transformation Prescription. 
\end{prop}

These results illustrate the range of geometric behaviors that may occur in the
presence of signature change, and they motivate the systematic analysis of
signature-type changing metrics developed in the following section.

\section{Non-orientable models and signature change}

Based on the method introduced in~\cite{Hasse + Rieger-Transformation},
any Lorentzian manifold $(M,g)$ can be transformed into a signature-changing
manifold with metric $\tilde{g}$ by means of the Transformation Prescription~(\ref{eq:Transformation Prescription}), using a non-vanishing timelike
line element field $\{V,-V\}$ together with a smooth function $f\colon M\longrightarrow\mathbb{R}$.
Ideally, one assumes the existence of a smooth, non-vanishing timelike
vector field $V$. Motivated by this approach, we analyze certain
classes of signature-changing metrics and investigate whether they
can be expressed in the form
\begin{equation}
\tilde{g}=g+fV^{\flat}\otimes V^{\flat},\label{eq:Transformation Prescription}
\end{equation}
where $f$ is a smooth transformation function interpolating between
the Lorentzian and Riemannian regions. The hypersurface of signature
change then appears as $\mathcal{H}=f^{-1}(1)$.

\subsection{The Lorentzian background: rotating Minkowski metric\label{subsec:Rotating-Minkowski-Metric}}

Let $M=\mathbb{R}^{2}$ with the standard topology and the canonical
Minkowski metric 
\[
\eta=-(dt)^{2}+(dx)^{2},
\]
where $(t,x)$ are Cartesian coordinates and let $E_{0}$, $E_{1}$ denote
the Gaussian basis vector fields. We restrict attention to two-dimensional
models, although additional flat directions could be included without
affecting curvature (corresponding to a flat plane with a vanishing
Riemannian curvature tensor).

~

To describe a rotation of the standard coordinate vector fields $(E_{0},E_{1}$)
clockwise by an angle $\varphi$, we define the new, non-holonomic
basis 

\[
F_{0}=-(\sin\varphi)E_{1}+(\cos\varphi)E_{0},\quad F_{1}=(\cos\varphi)E_{1}+(\sin\varphi)E_{0},
\]
where $\varphi$ is a smooth function of $(t,x)$. We now require
that the metric 
\[
g=g_{00}(dt)^{2}+2g_{01}dtdx+g_{11}(dx)^{2}
\]
 satisfies 
\[
g(F_{1},F_{1})=-g(F_{0},F_{0})=1,\quad g(F_{0},F_{1})=0.
\]

\textbf{~}\\ Solving this nonhomogeneous linear equation system gives
the unique solution

\[
g_{11}=-g_{00}=\cos^{2}\varphi-\sin^{2}\varphi=\cos(2\varphi),\text{ and }g_{01}=2\cos\varphi\sin\varphi=\sin(2\varphi).
\]

\textbf{~}\\ Hence the metric takes the form

\[
g=-\cos(2\varphi)(dt)^{2}+2\sin(2\varphi)dtdx+\cos(2\varphi)(dx)^{2},
\]
a smooth Lorentzian metric on $\mathbb{R}^{2}$ whose light cones
rotate clockwise by an angle $\varphi$. We refer to this
informally as the ``rotating Minkowski'' metric. Note, however,
that unlike Minkowski space, $(M,g)$ is not flat. 

\begin{rem}
The same metric can be obtained by starting with a non-symmetric
tensor and extracting its symmetric part. Explicitly, one may set
\end{rem}

\[
g=\ensuremath{-[(\cos\varphi)dx+(\sin\varphi)dt]\otimes[-(\sin\varphi)dx+(\cos\varphi)dt]}
\]
\[
=(\sin\varphi)(\cos\varphi)dx\otimes dx-(\cos^{2}\varphi)dx\otimes dt+(\sin^{2}\varphi)dt\otimes dx-(\sin\varphi)(\cos\varphi)dt\otimes dt,
\]
and then replace $g$ by its symmetrization $g_{(ab)}=\frac{1}{2}(g_{ab}+g_{ba})$
by producing a new symmetric tensor $g_{(ab)}$ from the old one,
based on 
\[
\underset{g}{\underbrace{g_{ab}}}=\underset{symmetric}{\underbrace{g_{(ab)}}}+\underset{skew-symmetric}{\underbrace{g_{[ab]}}}=\frac{1}{2}(g_{ab}+g_{ba})+\frac{1}{2}(g_{ab}-g_{ba}).
\]
The determinant of $g$ is $\det([g_{\mu\nu}])=-\cos^{2}(2\varphi)-\sin^{2}(2\varphi)=-1$,
which is negative and non-vanishing. Therefore, $g$ is Lorentzian.
A natural choice is $\varphi=\pi x$, which yields a metric in which
the light cones rotate clockwise in the $tx$-plane as one moves
in the $x$-direction. At integer values $x=k,k\in\mathbb{Z}$,
the metric reduces to the canonical Minkowski form $g=-1(dt)^{2}+1(dx)^{2}$.

~ 

Relative to an initial choice of time orientation, the causal cones
swap past and future orientation at odd integers $x\in\{2k-1:k\in\mathbb{Z}\}$,
while they realign with their initial orientation at even integers.
Thus, a full rotation of the light cone (with respect to the basis
vectors $E_{0}$ and $E_{1}$) occurs after an additional displacement
of the absolute value $1$ along the $x$-direction. Hence, a full
rotation occurs only at $x\in\{2k:k\in\mathbb{Z}\}$, where the past
and future of the light cones align with their orientation at $x=0$,
i.e. along the $t$-axis. The Lorentzian manifold $(M,g)$ is nevertheless
both orientable and time-orientable (see Subsection~\ref{subsec:Transformation-Prescription for non-compact}
for the definition of the global non-vanishing, timelike vector field
$V$).

\subsubsection{Causal Structure and Causal Trapping} 

The causal structure can be elucidated by means of the
null directions obtained from

\[
g=0 
 \overset{\cdot\frac{1}{(dx)^{2}}}{\Longleftrightarrow}g=-\cos(2\varphi)\frac{(dt)^{2}}{(dx)^{2}}+2\sin(2\varphi)\frac{dt}{dx}+\cos(2\varphi)=0,
\]
a quadratic equation for $\frac{dt}{dx}$ of the form $ax^{2}+bx+c=0$.
The two solutions are 

\[
\frac{dt}{dx}=\begin{cases}
\begin{array}{c}
\frac{-1+\sin(2\varphi)}{\cos(2\varphi)}=\frac{\sin(\varphi)-\cos(\varphi)}{\sin(\varphi)+\cos(\varphi)}\\
\frac{1+\sin(2\varphi)}{\cos(2\varphi)}=\frac{\sin(\varphi)+\cos(\varphi)}{\cos(\varphi)-\sin(\varphi)}
\end{array} & \Longleftrightarrow dt=\begin{cases}
\begin{array}{c}
\frac{\sin(\varphi)-\cos(\varphi)}{\sin(\varphi)+\cos(\varphi)}dx\\
\frac{\sin(\varphi)+\cos(\varphi)}{\cos(\varphi)-\sin(\varphi)}dx
\end{array} & .\end{cases}\end{cases}
\]

\textbf{~}\\ For $\varphi=\pi x$, these integrate to

\[
\begin{array}{c}
t(x)=-\frac{1}{\pi}\log\left|\sin(\pi x)+\cos(\pi x)\right|+\textrm{const,}\\
t(x)=-\frac{1}{\pi}\log\left|\cos(\pi x)-\sin(\pi x)\right|+\textrm{const.}
\end{array}
\]

\textbf{~}\\The Christoffel symbols can be expressed as periodic functions of
$x$. The non-vanishing Christoffel symbols are given by

~

$\Gamma_{tt}^{t}=-\frac{1}{2}g^{tx}\frac{d}{dx}g_{tt}=-\pi\sin^{2}(2\varphi)$,

$\Gamma_{xx}^{x}=-\Gamma_{tx}^{t}=-\Gamma_{tt}^{x}=\pi\sin(2\varphi)\cos(2\varphi)=\frac{\pi}{2}\sin(4\varphi)$, 

$\Gamma_{tx}^{x}=-\Gamma_{tt}^{t}=\pi\sin^{2}(2\varphi)=\frac{\pi}{2}(1-\cos(4\varphi))$,

$\Gamma_{xx}^{t}=-\pi(1+\cos^{2}(2\varphi))=-\frac{\pi}{2}(3+\cos(4\varphi))$.

\textbf{~}\\ All Christoffel symbols are linear combinations of $\sin(4\varphi)$
and $\cos(4\varphi)$, hence are periodic in $x$ with period $1/2$.
The scalar curvature in two dimensions (where it completely
characterizes the curvature) is given by 

\[
R=2K=\ensuremath{2(\varphi''\sin(2\varphi)+2\varphi'^{2}\cos^{3}(2\varphi)),}
\]
where $\varphi' := \frac{d\varphi}{dx} = \pi$, $\varphi'' = 0$, and
$K$ denotes the Gaussian curvature, which is bounded.

~

At first glance, $(M,g)$ looks deceivingly innocent. Moreover, the
underlying $2$-manifold is $M=\mathbb{R}^{2}$ with the standard
topology, so we are able to choose an atlas that consists of a single
chart equipped with the identity map. The causal structure is revealed
by the Killing vector fields. Since $g$ is independent of the coordinate
$t$, the vector field ${\displaystyle W=\partial_{t}}$ is a Killing
field (Figure~\ref{RMM-1} provides a hint about the appearance of
the Killing vector fields). However, its causal character varies periodically
with $x$: it is timelike on intervals 
\[
\frac{1}{4}(4\pi k-\pi)<\varphi<\frac{1}{4}(4\pi k+\pi)\Longleftrightarrow\frac{1}{4}(4k-1)<x<\frac{1}{4}(4k+1),\;k\in\mathbb{Z},
\]
becomes null at $x=\frac{k}{2}-\frac{1}{4}$, $k\in\mathbb{Z}$, and
spacelike on alternating intervals, $\frac{1}{4}(4k+1)<x<\frac{1}{4}(4k+3),k\in\mathbb{Z}$.
Accordingly, $(M,g)$ has a \textquotedblleft stripe-like\textquotedblright{}
causal structure: stationary regions $M(k)\coloneqq\{(t,x)\in M;\frac{1}{4}(4k-1)<x<\frac{1}{4}(4k+1)\},k\in\mathbb{Z}$,
where $\partial_{t}$ is timelike, alternate with non-stationary regions
$\{(t,x)\in M;\frac{1}{4}(4k+1)<x<\frac{1}{4}(4k+3)\}$ for $k\in\mathbb{Z}$,
separated by lightlike curves. Moreover, the scalar curvature $R$
is positive in stationary regions and negative in non-stationary regions.

\begin{figure}[H]
    \centering
    \includegraphics[scale=0.99]{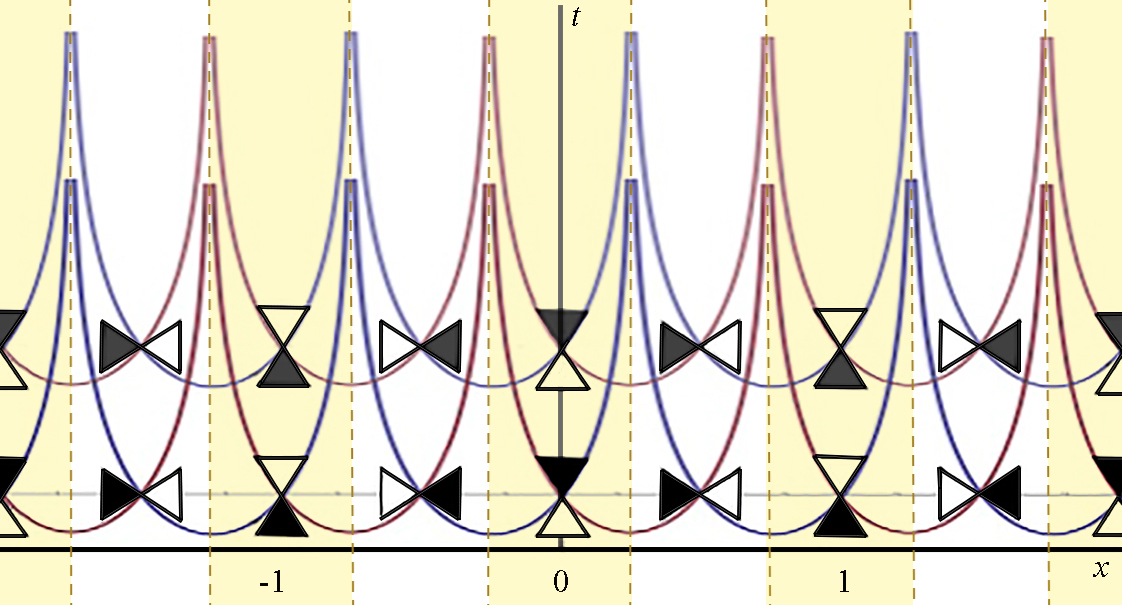}
    \caption{\small \label{RMM-1} The ``rotating Minkowski'' metric 
    $g=-\cos(2\varphi)\,(dt)^{2}+2\sin(2\varphi)\,dt\,dx+\cos(2\varphi)\,(dx)^{2}$
    on $\mathbb{R}^{2}$. The manifold has a ``stripe-like causal pattern''
    consisting of alternating stationary and non-stationary stripes (regions),
    such that adjacent stripes are separated by lightlike curves.}
\end{figure}

Figure~\ref{RMM-1} illustrates this alternating stripe pattern,
with stationary regions (yellow) bounded by lightlike curves (dashed).
If $\partial_{t}$ in these regions is taken to be future-pointing, then causal curves
that enter the stationary regions $M(2k)$, $k\in\mathbb{Z}$, become
trapped there, whereas no causal curve can enter the stationary regions
$M(2k-1)$, $k\in\mathbb{Z}$. We summarize the preceding discussion in the following Proposition~\ref{prop:trapped curves}.

\begin{prop*}[One-way causal barriers]
Let $(M,g)$ be the two-dimensional manifold $M=\mathbb{R}^2$ with coordinates $(t,x)$
equipped with the ``rotating Minkowski'' metric
\[
g = -\cos(2\varphi)\,(dt)^2
    + 2\sin(2\varphi)\,dt\,dx
    + \cos(2\varphi)\,(dx)^2,
\qquad \varphi=\pi x.
\]
For each $k\in\mathbb{Z}$ define the stationary stripe
\[
M_k:=\left\{(t,x)\in M:
k-\frac14 < x < k+\frac14
\right\}.
\]
Then $\partial_t$ is a Killing field and is timelike on each $M_k$.

Let the time orientation be induced by the global timelike vector field
\[
V=(\cos\varphi)\,\partial_t-(\sin\varphi)\,\partial_x .
\]
With respect to this time orientation, the null boundary components of the
stationary stripes are one-way causal barriers. More precisely, every
future-directed causal curve which enters an even stationary stripe $M_{2k}$
cannot leave it, whereas no future-directed causal curve can enter an odd
stationary stripe $M_{2k-1}$.
\end{prop*}

\begin{proof}
The metric coefficients are independent of the coordinate $t$, hence
\[
K:=\partial_t
\]
is a Killing vector field. Moreover,
\[
g(K,K)=g(\partial_t,\partial_t)=-\cos(2\varphi)
=-\cos(2\pi x).
\]
Thus \(K\) is timelike precisely where \(\cos(2\pi x)>0\), i.e. on the open
stripes
\[
M_k=\left\{(t,x)\in M:
k-\frac14 < x < k+\frac14
\right\}.
\]

It remains to analyze the causal character of the boundary components of these
stripes. The boundary of \(M_k\) consists of the two null lines
\[
L_k:=\left\{x=k-\frac14\right\},
\qquad
R_k:=\left\{x=k+\frac14\right\}.
\]
Let
\[
u=a\,\partial_t+b\,\partial_x
\]
be the tangent vector of a future-directed causal curve. We use the global time
orientation determined by
\[
V=(\cos\varphi)\partial_t-(\sin\varphi)\partial_x,
\]
so that \(u\) is future-directed if
\[
g(V,u)\leq 0.
\]

First consider the right boundary \(R_k\), where \(x=k+\frac14\). Then
\[
\sin(2\pi x)=1,
\qquad
\cos(2\pi x)=0.
\]
Hence
\[
g(u,u)=2ab.
\]
The causal condition \(g(u,u)\leq 0\) therefore gives
\[
ab\leq 0.
\]
Moreover, at \(R_k\) we have
\[
V=\frac{(-1)^k}{\sqrt2}\bigl(\partial_t-\partial_x\bigr),
\]
and hence
\[
g(V,u)=\frac{(-1)^k}{\sqrt2}(b-a).
\]
Thus future-directedness gives
\[
(-1)^k(b-a)\leq 0.
\]

\textbf{~}\\ If \(k\) is even, this implies \(b\leq a\). Together with \(ab\leq0\), it follows
that \(b\leq0\). Thus, at the right boundary of an even stripe, a future-directed
causal curve can cross only in the direction of decreasing \(x\), i.e. into the
stripe, not out of it.

\textbf{~}\\ If \(k\) is odd, future-directedness gives \(a\leq b\). Together with \(ab\leq0\),
this implies \(b\geq0\). Hence, at the right boundary of an odd stripe, a
future-directed causal curve can cross only in the direction of increasing \(x\),
i.e. out of the stripe, not into it.

Now consider the left boundary \(L_k\), where \(x=k-\frac14\). Then
\[
\sin(2\pi x)=-1,
\qquad
\cos(2\pi x)=0.
\]
Hence
\[
g(u,u)=-2ab,
\]
and the causal condition gives
\[
ab\geq0.
\]
At \(L_k\) we have
\[
V=\frac{(-1)^k}{\sqrt2}\bigl(\partial_t+\partial_x\bigr),
\]
and therefore
\[
g(V,u)=\frac{(-1)^k}{\sqrt2}(-a-b).
\]
Thus future-directedness gives
\[
(-1)^k(-a-b)\leq0.
\]

\textbf{~}\\ If \(k\) is even, this becomes \(a+b\geq0\). Together with \(ab\geq0\), it follows
that \(b\geq0\). Thus, at the left boundary of an even stripe, a future-directed
causal curve can cross only in the direction of increasing \(x\), i.e. into the
stripe, not out of it.

\textbf{~}\\ If \(k\) is odd, future-directedness gives \(a+b\leq0\). Together with \(ab\geq0\),
it follows that \(b\leq0\). Hence, at the left boundary of an odd stripe, a
future-directed causal curve can cross only in the direction of decreasing \(x\),
i.e. out of the stripe, not into it.

Consequently, for each even stripe \(M_{2k}\), both boundary components are
future one-way causal barriers directed inward. Hence any future-directed causal
curve entering \(M_{2k}\) cannot leave it. Conversely, for each odd stripe
\(M_{2k-1}\), both boundary components are future one-way causal barriers directed
outward. Hence no future-directed causal curve can enter \(M_{2k-1}\).
\end{proof}

\subsection{Non-Compact Models: The Infinite M\"{o}bius strip}

\subsubsection{General application of the Transformation Prescription\label{subsec:Transformation-Prescription for non-compact}}

As established in Subsection~\ref{subsec:Rotating-Minkowski-Metric},
our example $(M,g)$ is a non-compact, orientable, and time-orientable
Lorentzian manifold. Consequently, it admits a global, non-vanishing
timelike vector field. By an appropriate choice, we take 
\begin{equation}
V=(\cos\varphi)\frac{\partial}{\partial t}-(\sin\varphi)\frac{\partial}{\partial x}\label{eq: vectorfield}
\end{equation}
as a global timelike vector field with respect to the metric 
\begin{equation}
g=-\cos(2\varphi)(dt)^{2}+2\sin(2\varphi)dtdx+\cos(2\varphi)(dx)^{2},\quad\varphi=\pi x.\label{eq: metric}
\end{equation}
A direct computation shows that $g(V,V)=-1$.

\textbf{~}\\ Given this vector field $V$, we consider the class
of signature-changing metrics

\[
\tilde{g}=g+fV^{\flat}\otimes V^{\flat},
\]
where $f$ is an arbitrary smooth function and $V^{\flat}=g(V,\centerdot)$
is the index-lowering (morphism) of $V\in\mathfrak{X}(M)$. The resulting
metric depends on the specific choice of $V$; however, in this case,
the above $V$ may be regarded as a canonical choice for $(M,g)$.
Carrying out the calculation yields

\[
\tilde{g}=g+f\cdot\left(g((\cos\varphi)\frac{\partial}{\partial t}-(\sin\varphi)\frac{\partial}{\partial x},\centerdot)\otimes g((\cos\varphi)\frac{\partial}{\partial t}-(\sin\varphi)\frac{\partial}{\partial x},\centerdot)\right)
\]

\[
\Longleftrightarrow\tilde{g}=(-\cos(2\varphi)(dt)^{2}+2\sin(2\varphi)dtdx+\cos(2\varphi)(dx)^{2})
\]

\[
+f\cdot\left(\cos^{2}(\varphi)(dt)^{2}-2\sin(\varphi)\cos(\varphi)dtdx+\sin^{2}(\varphi)(dx)^{2}\right),
\]
which explicitly takes the form

\[
\tilde{g}=(f\cdot\cos^{2}(\varphi)-\cos(2\varphi))(dt)^{2}+(2-f)\sin(2\varphi)dtdx+(f\cdot\sin^{2}(\varphi)+\cos(2\varphi))(dx)^{2}.
\]
Equivalently, in matrix representation:

\[
[\tilde{g}_{\mu\nu}]=\left(\begin{array}{cc}
f\cdot\cos^{2}(\varphi)-\cos(2\varphi) & (1-\frac{f}{2})\sin(2\varphi)\\
(1-\frac{f}{2})\sin(2\varphi) & f\cdot\sin^{2}(\varphi)+\cos(2\varphi)
\end{array}\right).
\]

\textbf{~}\\ The determinant of the metric is easily computed as 
\[
\det([\tilde{g}_{\mu\nu}])=f-1,
\]
so the metric $\tilde{g}$ is degenerate precisely at those points $p$ for which $f(p)=1$. Hence, the hypersurface of signature change is the $(n-1)$-dimensional
submanifold 
\[
\mathfrak{\mathcal{H}:=}f^{-1}(1)=\{p\in M\colon f(p)=1\}.
\]

\textbf{~}\\ To obtain a transverse type-changing semi-Riemannian
manifold, we must impose the condition 
\begin{equation}
d(\det([\tilde{g}_{\mu\nu}]))=d(f-1)\neq0,\quad\forall q\in\mathcal{H},\label{eq: condition1}
\end{equation}
in every local coordinate system $\xi=(x^{0},\ldots,x^{n-1})$ around
$q$.  
\textbf{~}\\In the present example this simply means
\[
d(f-1)=\frac{\partial f(t,x)}{\partial t}dt+\frac{\partial f(t,x)}{\partial x}dx=df\neq0, \qquad \forall  q=(t,x)\in\mathcal{H}.
\]
In addition, the radical must not be tangent to $\mathcal{H}$
for any $q\in\mathcal{H}$. Since for each $q\in\mathcal{H}$ the tangent space $T_{q}\mathcal{H}$ is the
kernel of the map $df_{q}\colon T_{q}M\longrightarrow T_{1}\mathbb{R}$, the condition 
\begin{equation}
(V(f))(q)=((df)(V))(q)\neq0\qquad\text{for all } q\in\mathcal{H}\label{eq: condition2}
\end{equation}
ensures that $V\notin T_{q}\mathcal{H}$, and therefore $V$ is nowhere tangent to $\mathcal{H}$.

\begin{rem}
To ensure the transversality of $\mathrm{Rad}_{q}$, one must have
$V(f)\neq 0$ for all points $p\in\mathcal{H}$. With the given vector
field $V$ (see Equation~(\ref{eq: vectorfield})), this condition becomes, for every
$p\in\mathcal{H}$,
\[
(df)\!\left((\cos\varphi)\frac{\partial}{\partial t}
           -(\sin\varphi)\frac{\partial}{\partial x}\right)\neq 0
\quad\Longleftrightarrow\quad
(\cos\varphi)\frac{\partial f}{\partial t}
\neq
(\sin\varphi)\frac{\partial f}{\partial x}.
\]
Since $f$ is constant along $\mathcal{H}$, with $f(t,x)=1$, we obtain
there (assuming the transversality of the radical)
\[
\frac{dx}{dt}
  = \frac{\frac{\partial f}{\partial t}}
         {\frac{\partial f}{\partial x}}
  \neq -\tan(\varphi)
  = -\tan(\pi x).
\]

Whether this can be achieved by a suitable choice of $f$ in the cases
under consideration depends, among other things, on the required
periodicity conditions. One possible approach is to begin with a purely
tangential radical and solve the differential equation
\[
\frac{dx}{dt}=-\tan(\pi x).
\]
Integrating yields
\[
t=-\int\cot(\pi x)\,dx
  = -\frac{1}{\pi}\ln\lvert \sin(\pi x)\rvert + \mathrm{const}.
\]
To ensure that this relation is never satisfied, one may shift the graph
of the function $t$ (considered as a function of $x$) slightly in the
$x$-direction, obtaining
\[
t=-\frac{1}{\pi}\ln\lvert \sin(\pi(x+\varepsilon))\rvert
  + \mathrm{const}.
\]

In this situation, the Conditions~(\ref{eq: condition1}) and~(\ref{eq: condition2}) reduce simply to the
requirement that
\[
df\neq 0 \qquad\text{for all } q\in\mathcal{H}.
\]
\end{rem}

\subsubsection{Construction of the signature-changing infinite M\"{o}bius strip \label{subsec:Open-Infinite-Moebius}}

We again take $M=\mathbb{R}^{2}$ with the standard topology and the
metric 
\[
g=-\cos(2\varphi)(dt)^{2}+2\sin(2\varphi)dtdx+\cos(2\varphi)(dx)^{2},\quad\varphi=\pi x,
\]
as in Subsection~\ref{subsec:Rotating-Minkowski-Metric}. 

\textbf{~}\\ The M\"{o}bius strip need not be time-orientable; this
depends on the identification of opposite edges. To obtain a time-orientable
model, consider $M=\mathbb{R}\times\mathbb{R}$ (the infinite strip
of width $(-\infty,\infty)$), equipped with the
standard topology and the metric $g$. Then we fold the manifold by
defining the quotient map 
\[
q\colon\mathbb{R}\times\mathbb{R}\longrightarrow(\mathbb{R}\times\mathbb{R})/\sim,\qquad(t,x)\sim(\tilde{t},\tilde{x})\Longleftrightarrow(\tilde{t},\tilde{x})=((-1)^{k}t,x+k),\quad k\in\mathbb{Z}.
\]
The resulting Lorentzian manifold 
\[
M_{\infty}\coloneqq(M/\sim,g)
\]
is non-compact, non-orientable, but time-orientable, and has the topology
of an open infinite M\"{o}bius strip, globally distinct from Minkowski
space.

\textbf{~}\\ Since $M_{\infty}$ is time-orientable, it admits a
global timelike vector field 
\[
V=(\cos\varphi)\frac{\partial}{\partial t}-(\sin\varphi)\frac{\partial}{\partial x}.
\]
Hence, by applying the Transformation Prescription~(\ref{eq:Transformation Prescription}),
we obtain the signature-changing M\"{o}bius strip 
\begin{equation}
\tilde{M}_{\infty}=(M/\sim,\tilde{g}),\label{eq:signature-changing Moebius strip}
\end{equation}
where
\begin{equation}
\tilde{g}=(f\cos^{2}(\varphi)-\cos(2\varphi))(dt)^{2}+(2-f)\sin(2\varphi)dtdx+(f\sin^{2}(\varphi)+\cos(2\varphi))(dx)^{2},\label{eq:signature-changing Moebius metric}
\end{equation}
with the same quotient identification as above, $(\tilde{t},\tilde{x})=((-1)^{k}t,x+k)$,
$k\in\mathbb{Z}$, and $f\in C^{\infty}(M)$ determining the hypersurface
of signature change $\mathfrak{\mathcal{H}=}f^{-1}(1)$. 

\begin{figure}[H]
\centering{}\includegraphics[scale=0.75]{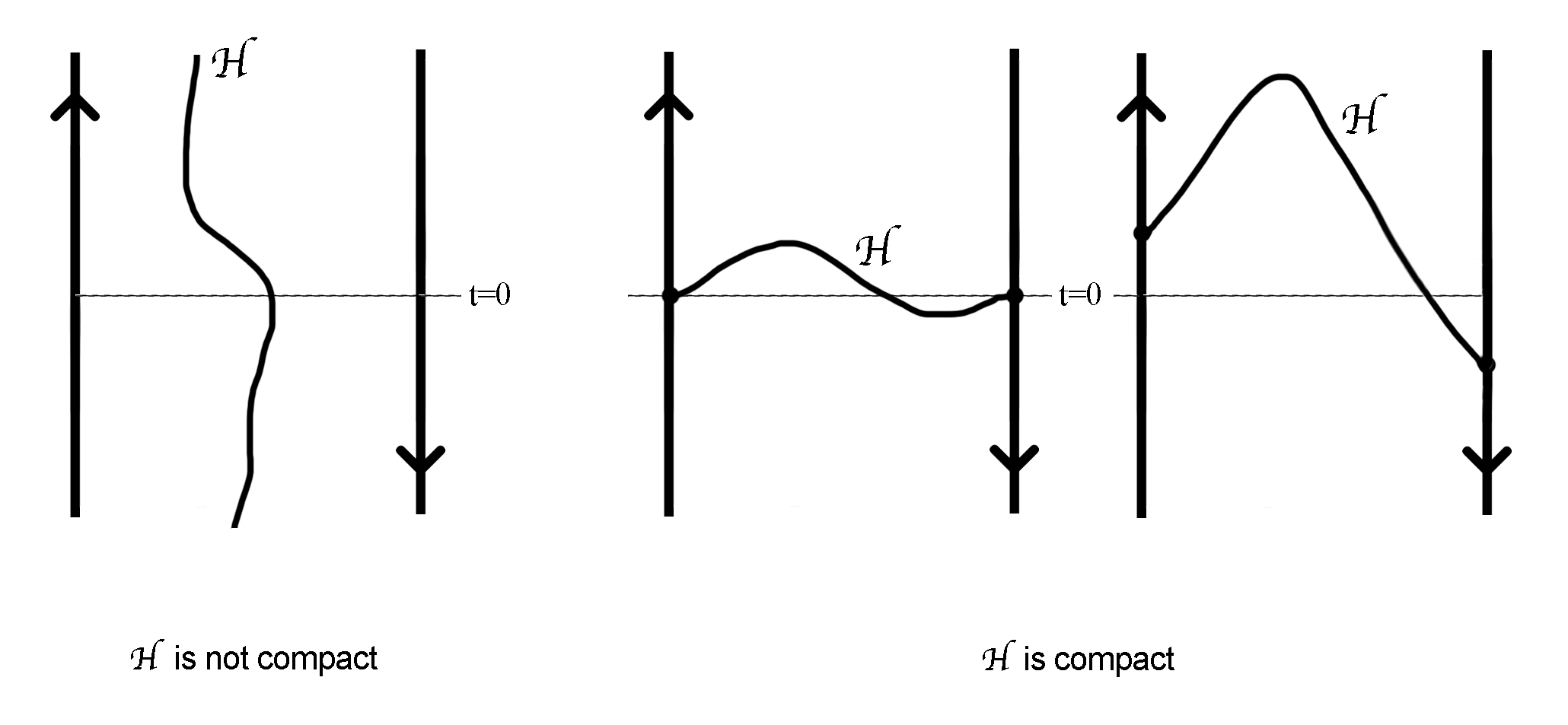}\caption{{\small{}\label{fig:hypersurface-alignment}The hypersurface of signature
change can be either compact or non-compact: The very left example
depicts a hypersurface that is not compact. The two right examples
show a hypersurface that is compact. For the latter two, it should
be noted that the inclinations of the tangents at $\mathcal{H}$ must
also align with the identification.}}
\end{figure}

\begin{rem}
Since only the $x$-direction of $M$ is compactified,
while the $t$-direction remains unbounded, the hypersurface $\mathcal{H}$
may be compact or non-compact depending on its orientation relative
to the identification (see Figure~\ref{fig:hypersurface-alignment}).
In particular, the two right illustrations show compact hypersurfaces,
while the left illustrates a non-compact one. For compact cases, the
tangents of $\mathcal{H}$ must align with the identification.
\end{rem}

The above reasoning provides an alternative proof of Proposition 6
 in~\cite{Hasse + Rieger-Loops}. For ease of reference, we restate Proposition~\ref{prop:orientable+timeorientable-not_orientable}:
\begin{prop*}
Even if a transverse, signature-type changing manifold $(M,g)$ with
a transverse radical is pseudo-time orientable and pseudo-space orientable,
it is not necessarily orientable.
\end{prop*}

\begin{proof}
Consider the $2$-dimensional M\"{o}bius
strip $M_{\infty}:=(M/\sim,g)$ with the standard topology and the
metric $g=-\cos(2\varphi)(dt)^{2}+2\sin(2\varphi)dtdx+\cos(2\varphi)(dx)^{2}$,
$\varphi=\pi x$ (Subsection~\ref{subsec:Open-Infinite-Moebius}).
Since $M_{\infty}$ is time-orientable, it admits a global timelike
vector field 
\[
V=(\cos\varphi)\frac{\partial}{\partial t}-(\sin\varphi)\frac{\partial}{\partial x},\qquad g(V,V)=-1.
\]
We extend this $V$ to an orthonormal basis $\{V,E_{1}\}$ of $T_{p}M$
at each $p\in M$, with $g(E_{1},E_{1})=1$ and $g(V,E_{1})=0$, with
$E_{1}$ being a frame field.

\textcolor{black}{~\\ Applying the Transformation Prescription yields
the signature-changing }M\"{o}bius\textcolor{black}{{} strip $\tilde{M}_{\infty}=(M/\sim,\tilde{g})$,
where $\tilde{g}$ is as above in Equations~\ref{eq:signature-changing Moebius strip}
and~\ref{eq:signature-changing Moebius metric}. ~\\ The manifold
$\tilde{M}_{\infty}$ is pseudo-time orientable, with $V=(\cos\varphi)\frac{\partial}{\partial t}-(\sin\varphi)\frac{\partial}{\partial x}$
being a continuous, non-vanishing, timelike vector field $V$ on $M_{L}$.
And also with $\{E_{1}\}$ spacelike, $\tilde{M}_{\infty}$ is pseudo-space
orientable, yet the }M\"{o}bius\textcolor{black}{{} strip $\tilde{M}_{\infty}$
is not orientable.} \end{proof}

These constructions demonstrate how the Transformation Prescription
extends naturally from orientable Lorentzian manifolds to non-orientable
settings such as the\textcolor{black}{{} }M\"{o}bius strip. In particular,
while orientability is lost under the quotient, pseudo-time and pseudo-space
orientability persist, illustrating that the local structure governed
by the Transformation Theorem remains intact even in globally non-orientable
manifolds. This positions the\textcolor{black}{{} }M\"{o}bius strip
as a canonical model for exploring the interplay between local signature
change and global topological obstructions.

\subsection{Compact models and topological obstructions}

It is well known that every non-compact smooth manifold admits a Lorentzian
metric. However, this is not generally the case for compact manifolds.
Moreover, any compact time-oriented Lorentzian manifold necessarily
possesses a non-empty chronology-violating set. In particular, such
manifolds contain closed timelike curves, which are typically regarded
as pathological.

~

In this subsection we examine two compact $2$-manifolds: the compact
M\"{o}bius strip, which admits a Lorentzian metric, and the real
projective plane, which does not. Throughout, all manifolds are assumed
to be smooth and may have non-empty boundary. By Brown\textquoteright s
collaring theorem~\cite{Brown - Collar Theorem,Lee - Introduction to Smooth Manifolds},
the boundary $\mathcal{H}$ admits a collar neighborhood.
We further assume that there exists a diffeomorphism between the relevant
collar neighborhoods.

\subsubsection{Compact M\"{o}bius strip with boundary\label{subsec:Compact Moebius-Strip}}

We now consider the quotient manifold with boundary 

\begin{equation}
M=([0,1]\times[0,1])/\sim,\label{eq:compact Moebius}
\end{equation} 
where the equivalence relation identifies opposite sides of the unit
square via

\[
(t,0)\sim(1-t,1).
\]
 This construction is illustrated in Figure~\ref{fig:manifold with boundary}.

\textbf{~}\\ Topologically, this quotient is the compact M\"{o}bius strip. It is equipped with the rotating
metric

\[
g=-\cos(2\varphi)(dt)^{2}+2\sin(2\varphi)dtdx+\cos(2\varphi)(dx)^{2},\quad\varphi=\pi x,
\]
as induced from the metric on $M_{\infty}$, where the quotient map
identifies $(t,0)$ with $(1-t,1)$, see Subsection~\ref{subsec:Rotating-Minkowski-Metric}.

\textbf{~}\\ The resulting M\"{o}bius strip 
\[
(M,g)=(([0,1]\times[0,1])/\sim,g)
\]
is compact, not simply connected, non-orientable, and time-orientable.
Thus it has the topology of the M\"{o}bius strip.\footnote{By contrast, under the identification $(t,x)\sim(-t,x+3)$, one obtains
a non-time-orientable M\"{o}bius strip.} 

\begin{figure}[H]
\centering{}\includegraphics[scale=1.11]{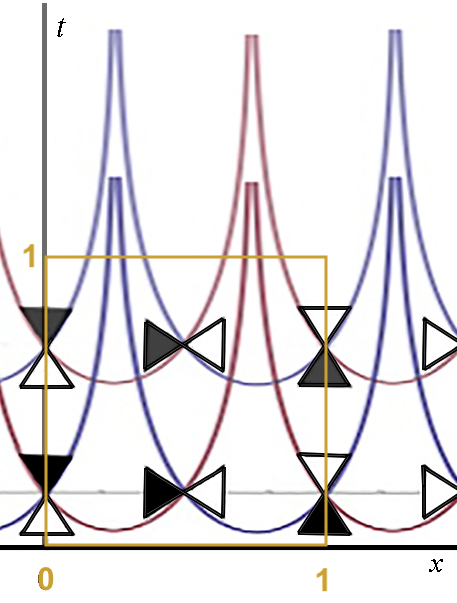}\caption{{\small{}\label{fig:manifold with boundary}The yellow unit square represents the quotient manifold with boundary $M=([0,1]\times[0,1])/\sim$,
obtained by identifying the opposite sides via $(t,0)\sim(1-t,1)$.}}
\end{figure}
\textbf{~}\\As a time-orientable Lorentzian manifold, the compact
M\"{o}bius strip $M$ admits a global, non-vanishing vector field
\[
V=(\cos\varphi)\frac{\partial}{\partial t}-(\sin\varphi)\frac{\partial}{\partial x}.
\]
Therefore, we may apply the Transformation Prescription~(\ref{eq:Transformation Prescription})
to obtain the signature-type changing M\"{o}bius strip 
\[
(\tilde{M},\tilde{g})=(([0,1]\times[0,1])/\sim,\tilde{g}),
\]
where 
\[
\tilde{g}=(f\cdot\cos^{2}(\varphi)-\cos(2\varphi))(dt)^{2}+(2-f)\sin(2\varphi)dtdx+(f\cdot\sin^{2}(\varphi)+\cos(2\varphi))(dx)^{2},
\]
and $f$ is an arbitrary $C^{\infty}$ function determining the hypersurface
of signature change $\mathfrak{\mathcal{H}=}f^{-1}(1)$.~\\{}

\subsubsection{The crosscap manifold}

Consider again the compact M\"{o}bius strip 
\[
M=([0,1]\times[0,1])/((t,0)\sim(1-t,1))
\]
with nonempty boundary $\partial M$. Note that 
\[
M_{\frac{1}{2}}=(\{\frac{1}{2}\}\times[0,1])/((\frac{1}{2},0)\sim(\frac{1}{2},1))\subset M
\]
represents the center line of the M\"{o}bius strip. The boundary,
which is a $1$-dimensional manifold, is given by 
\[
\partial M=\{(0,x)\colon x\in[0,1]\}\cup\{(1,x)\colon x\in[0,1]\},
\]
with identifications $(0,1)\sim(1,0)$ and $(0,0)\sim(1,1)$. Hence,
$\partial M\cong S^{1}$.

\textbf{~}\\ We can therefore define a map 
\[
\phi\colon[0,\frac{1}{2}]\times S^{1}\longrightarrow M,
\]
such that $\phi(0,\tilde{x})\in\partial M$ and $\phi(\frac{1}{2},\tilde{x})\in M_{\frac{1}{2}}$.
For each $(t,\tilde{x})\in[0,\frac{1}{2}]\times S^{1}$, the point
$\phi(t,\tilde{x})\in M$ is obtained by an orthogonal displacement
of length $t$ into the M\"{o}bius strip, starting at $\tilde{x}\in\partial M$.\footnote{The term ``length'' here does not refer to any metric on 
$M$, but only to differences in the $t$-coordinate.}
The map $\phi$ is continuous and defined over the entire M\"{o}bius
strip, making it surjective. Since antipodal points, $\pm\tilde{x}\in S^{1}$,
on the boundary satisfy 
\[
\phi(\frac{1}{2},\tilde{x})=\phi(\frac{1}{2},-\tilde{x}),
\]
 we obtain the quotient 
\[
C\coloneqq([0,\frac{1}{2}]\times S^{1})/\sim,\qquad(\frac{1}{2},\tilde{x})\sim(\frac{1}{2},-\tilde{x}),
\]
which is the crosscap. Note that we can now define the boundary as
$\partial M\coloneqq\phi(\{0\}\times S^{1})$. Thus, after identifying
antipodal points, we obtain a continuous bijection 
\[
\tilde{\phi}\colon C\longrightarrow M.
\]
\textcolor{black}{Since $C$ is compact and $M$ is Hausdorff, $\tilde{\phi}$
is a homeomorphism.}

\begin{figure}[H]
\centering{}\includegraphics[scale=0.37]{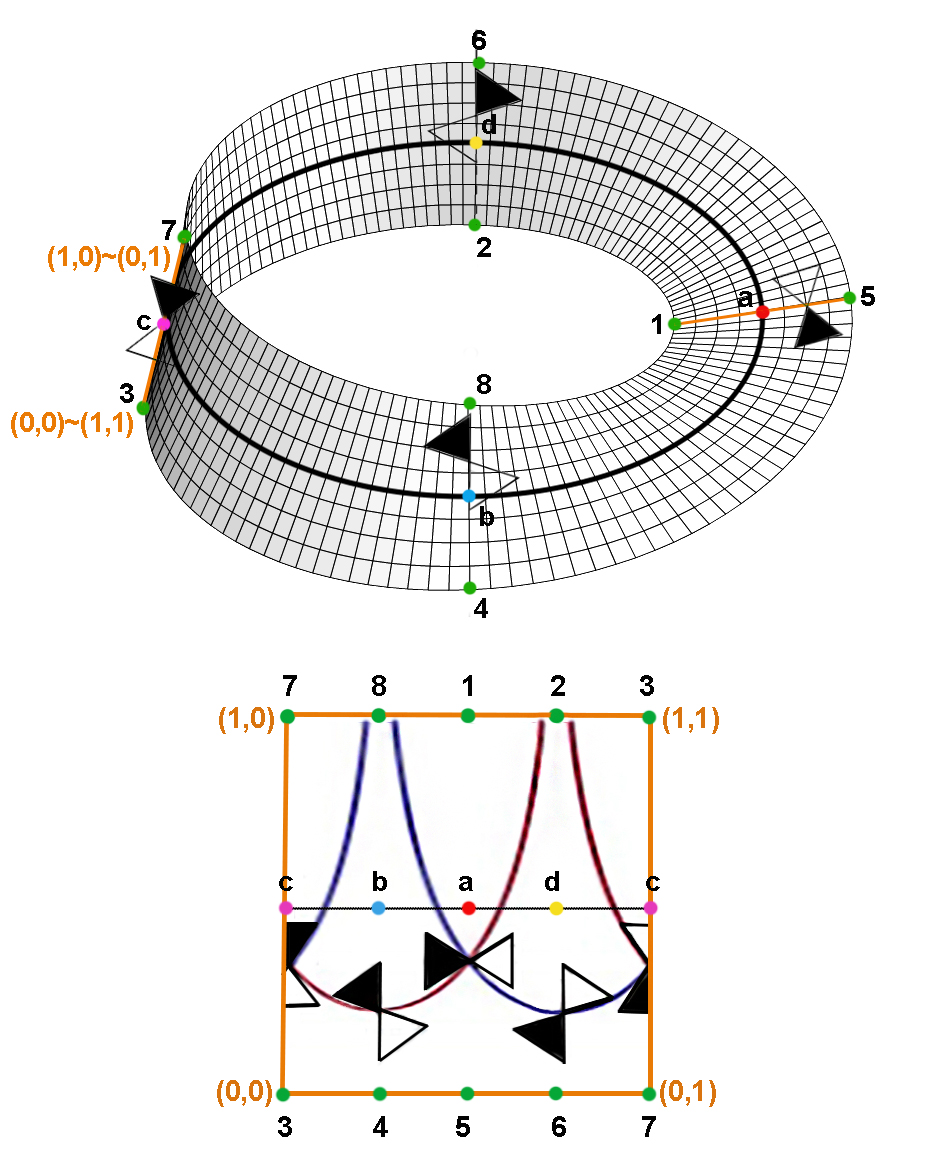}\caption{{\small{}\label{fig:Light-cones on Moebius}The light cone structure
in the M\"{o}bius strip.}}
\end{figure}

\textbf{~}\\ The crosscap is closed. However, a compact manifold
admits a Lorentzian metric only if its Euler characteristic is zero.\footnote{In this context, a closed manifold is a manifold without boundary
that is compact.} Unlike the M\"{o}bius strip, the crosscap does not admit
a globally defined Lorentzian metric. Therefore, the Transformation Prescription~(\ref{eq:Transformation Prescription}) cannot be applied to produce
a signature-changing metric directly from the crosscap. Nevertheless,
since the crosscap can be viewed as sewing a M\"{o}bius strip to
the boundary of a disk, we may instead refer to the compact M\"{o}bius
strip $M$ (Subsection~\ref{subsec:Compact Moebius-Strip}) and
use its Lorentzian metric $g$ for the Lorentzian sector of the crosscap,
see Figure~\ref{fig:Light-cones on Moebius}. The locus of signature
change is then naturally placed at $\partial M$, while the closed
unit disk $\mathbb{\overline{D}}^{2}$ constitutes the Riemannian
sector. By Brown\textquoteright s collaring theorem~\cite{Brown - Collar Theorem,Lee - Introduction to Smooth Manifolds},
both the M\"{o}bius strip and the disk have collar neighborhoods
along their common boundary.

\textbf{~}\\ To elucidate the partition of the crosscap into those
three areas, we formally describe the crosscap as the \textit{adjunction
space}~\cite{Lee - Introduction to Smooth Manifolds} 

\begin{equation}
C:=\overline{\mathbb{D}}^{2}\cup_{\psi}\mathbb{M}\label{eq: adjunction space}
\end{equation} 

\begin{equation}
\overline{\mathbb{D}}^{2}
= \{\, (\theta,r) : \theta \in S^{1},\ r \in [0,1] \,\},
\qquad
\partial\overline{\mathbb{D}}^{2}
= \{\, (\theta,1) : \theta \in S^{1} \,\}, 
\end{equation} 
with the attaching map
\[
\psi(\theta,1)=\begin{cases}
\begin{array}{c}
(1,\frac{\theta}{\pi}),\\
(0,\frac{\theta}{\pi}-1),
\end{array} & \begin{array}{c}
0\leq\theta\leq\pi,\\
\pi\leq\theta\leq2\pi.
\end{array}\end{cases}
\]


\textbf{~}\\ The map $\psi\colon\partial\overline{\mathbb{D}}^{2}\longrightarrow\partial M$
is continuous but not injective, and goes from the boundary circle
of the disk to the boundary of the M\"{o}bius strip. The locus of
signature change is then 
\[
\mathcal{H}\coloneqq\partial M=\{(0,x)\colon x\in[0,1]\}\cup\{(1,x)\colon x\in[0,1]\},
\]
with identifications $(0,1)\sim(1,0)$ and $(0,0)\sim(1,1)$. 

\textbf{~}\\ 
Throughout this section, we work under Condition~\ref{transverse radical}, which requires the radical to be transverse to the hypersurface of signature change. In particular, any admissible transformation function $f$ must produce a metric whose radical is nowhere tangent to $\mathcal H$. As we shall see, this requirement already leads to an obstruction on the crosscap, independently of the specific choice of $f$.

\textbf{~}\\ If we require the signature change to occur at $\partial M \cong S^{1}$, we may set
\[
f(t,x)=t^{2}+x^{2}
\]
as a simple, rotationally symmetric $C^{\infty}$ function satisfying $f^{-1}(1)=\partial M$, which induces a signature
change at the boundary $\partial M$ of $\tilde{M}_{c}=(([0,1]\times[0,1])/\sim,\tilde{g})$. Consequently,

\[
\mathcal{H}=\partial M=f^{-1}(1)=\{(t,x)\colon t^{2}+x^{2}=1\}=S^{1}.
\]
Inserting this into the transformed metric 

\[
\tilde{g}=(f\cdot\cos^{2}(\varphi)-\cos(2\varphi))(dt)^{2}+(2-f)\sin(2\varphi)dtdx+(f\cdot\sin^{2}(\varphi)+\cos(2\varphi))(dx)^{2},
\]
yields the associated matrix representation

\[
[\tilde{g}]=\left(\begin{array}{cc}
(t^{2}+x^{2})\cdot\cos^{2}(\pi x)-\cos(2\pi x) & -(t^{2}+x^{2}-2)\sin(\pi x)\cos(\pi x)\\
-(t^{2}+x^{2}-2)\sin(\pi x)\cos(\pi x) & (t^{2}+x^{2})\cdot\sin^{2}(\pi x)+\cos(2\pi x)
\end{array}\right),
\]
and we find that its determinant is 
\[
\det([\tilde{g}])=t^{2}+x^{2}-1.
\]

\textbf{~}\\ Thus $\tilde{g}$ is Lorentzian for $t^{2}+x^{2}=f(t,x)<1$
and Riemannian for $t^{2}+x^{2}=f(t,x)>1$. This metric $\tilde{g}$
reaches its canonical form $g=-1(dt)^{2}+1(dx)^{2}$ only for $x=t=0$.
Moreover, $d(\det([\tilde{g}]))=2tdt+2xdx=df\neq0$ for any $(t,x)\in\mathcal{H}$.

~\\ The causal structure can be elucidated by means of

~\\ $0=((t^{2}+x^{2})\cdot\cos^{2}(\pi x)-\cos(2\pi x))(\frac{dt}{dx})^{2}-2(t^{2}+x^{2}-2)\sin(\pi x)\cos(\pi x)\frac{dt}{dx}+((t^{2}+x^{2})\cdot\sin^{2}(\pi x)+\cos(2\pi x))$

~\\ $\Longleftrightarrow\frac{dt}{dx}=\frac{(t^{2}+x^{2}-2)\sin(\pi x)\cos(\pi x)\pm\sqrt{-t^{2}-x^{2}+1}}{((t^{2}+x^{2})\cdot\cos^{2}(\pi x)-\cos(2\pi x))}$.

~\\ The causal structure, along with the determinant $\det([\tilde{g}])=t^{2}+x^{2}-1<0\Longleftrightarrow t^{2}+x^{2}<1$,
indicates that the Lorentzian portion of the metric $\tilde{g}$ is
situated within the closed disk $\mathbb{\overline{D}}^{2}$. 

\textbf{~}\\ This situation already violates Condition~\ref{transverse radical}. Indeed, a direct computation shows that the radical of $\tilde g$ is tangent to $\mathcal H$ at isolated points. Consequently, the metric $\tilde g$ fails to be transverse at the locus of signature change and therefore does not belong to the class of transverse type-changing metrics considered in this work. The contradiction with the Poincar\'e--Hopf theorem provides an additional global obstruction, but it is not essential for excluding this example.


\textbf{~}\\ Even if we instead consider the signature transformation
\[
\tilde{g}=-(dt)^{2}+(dx)^{2}+f\left(g(\frac{\partial}{\partial t},\centerdot)\otimes g(\frac{\partial}{\partial t},\centerdot)\right)=(f-1)(dt)^{2}+(dx)^{2},
\]
with the global, non-vanishing, timelike vector field $V=\frac{\partial}{\partial t}$
with respect to $g=-(dt)^{2}+(dx)^{2}$, then with $f(t,x)=t^{2}+x^{2}$,
this yields

\[
\tilde{g}=(f-1)(dt)^{2}+(dx)^{2}=(t^{2}+x^{2}-1)(dt)^{2}+(dx)^{2}.
\]
The causal structure, along with the determinant, yields $\det([\tilde{g}])=t^{2}+x^{2}-1<0\Longleftrightarrow t^{2}+x^{2}<1$,
revealing that the Lorentzian portion of the metric $\tilde{g}$ is
located within the closed disk $\mathbb{\overline{D}}^{2}$. This
is again a contradiction!

\textbf{~}\\ We emphasize that replacing $f(t,x)$ by a monotone function of $t^{2}+x^{2}$ that interchanges the Lorentzian and Riemannian sectors does not remove this obstruction: the vector field $V$ remains unchanged, $V(f)$ only changes sign, and the radical remains tangent at the same points of $\mathcal H$ (i.e., the radical of $\tilde{g}$ is of mixed character with respect to $\mathcal{H}$). Consequently, $\tilde{g}$ does not
fall into the class of \textit{transverse type-changing metrics with
a transverse radical}, as defined in~\cite{Hasse + Rieger-Transformation, Hasse + Rieger-Local Transformation, Hasse + Rieger-Loops, Kossowksi + Kriele - Signature type change and absolute time in general relativity, Kossowski - The Volume BlowUp and Characteristic Classes for Transverse}. In particular, there is no smooth function
$f$ that produces a purely transverse signature change at $\mathcal{H}=\partial M$.
Hence the Transformation Theorem cannot be applied even locally to
the crosscap.

~

\begin{rem*}
If Condition~\ref{transverse radical} is dropped, one may allow the radical to be tangent to $\mathcal H$. In this broader setting, the above construction yields a well-defined signature-changing metric on the crosscap. However, the radical necessarily becomes tangent at isolated points of $\mathcal H$, and pseudo-time orientation fails to extend globally. Thus, even without imposing transversality, the crosscap exhibits intrinsic obstructions to a smooth causal structure compatible with a transformation prescription.
\end{rem*}

\subsubsection{The real projective plane ($\mathbb{R}P^{2}$) construction\label{subsec:Cross-Cap-metric}}

We now turn to the construction of a signature-type changing metric
on the real projective plane $\mathbb{R}P^{2}$. Recall that $M$ (see Equation~\ref{eq:compact Moebius})
is compact and $\partial M$ is a closed (and therefore compact) subspace
of $M$. The M\"{o}bius strip models considered here do not fall
into the framework of~\cite{Hasse + Rieger-Transformation, Hasse + Rieger-Local Transformation, Hasse + Rieger-Loops, Kossowksi + Kriele - Signature type change and absolute time in general relativity, Kossowski - The Volume BlowUp and Characteristic Classes for Transverse}, since the associated radical
distribution is partly transverse and partly tangent to the degeneracy
hypersurface. Consider the M\"{o}bius strip with boundary

\[
\overline{\mathbb{M}}=([-\sqrt{2},\sqrt{2}]\times[-\sqrt{2},\sqrt{2}])/\thicksim,
\]
where the identification is given by

\[
(t,-\sqrt{2})\sim(-t,\sqrt{2}).
\]
Its boundary is 
\[
\partial\mathbb{M}=\{(-\sqrt{2},x)\colon x\in[-\sqrt{2},\sqrt{2}]\}\cup\{(\sqrt{2},x)\colon x\in[-\sqrt{2},\sqrt{2}]\}
\]
with the additional identifications 
\[
(\sqrt{2},-\sqrt{2})\sim(-\sqrt{2},\sqrt{2}),\qquad(-\sqrt{2},-\sqrt{2})\sim(\sqrt{2},\sqrt{2}).
\]
Inside $\overline{\mathbb{M}}$, we distinguish the center line

\[
\mathbb{M}_{0}=(\{0\}\times[-\sqrt{2},\sqrt{2}])/\thicksim,\qquad(0,-\sqrt{2})\sim(0,\sqrt{2}),
\]
which is a closed curve in $\overline{\mathbb{M}}$.\textbf{~}\\ The
quotient manifold is equipped with the subspace topology of $\mathbb{R}^{2}$
and local Cartesian coordinates $(t,x)$. It is a standard fact that
the punctured projective plane is topologically equivalent to the
open M\"{o}bius strip:
\[
\mathbb{M}\cong\mathbb{R}P^{2}\setminus\{(t,x)\in\mathbb{R}^{2}\colon t^{2}+x^{2}\leq1\}=\mathbb{R}P^{2}\setminus\mathbb{\overline{D}}^{2}.
\]

\begin{rem}
Recall that the real projective plane $\mathbb{R}P^{2}$ can be obtained from the
unit sphere $S^{2} \subset \mathbb{R}^{3}$ by identifying antipodal points; that
is,
\[
\mathbb{R}P^{2} \cong S^{2}/\!\sim,
\]
where the equivalence relation is $p \sim -p$.  
The \emph{antipodal map}
\[
A \colon S^{2} \to S^{2}, \qquad A(p) = -p,
\]
is an isometry of $S^{2}$. Moreover, the canonical projection
\[
\pi \colon S^{2} \to \mathbb{R}P^{2}
\]
is the corresponding quotient map, identifying each point with its antipode.
Note that the tangent spaces $T_{p}S^{2}$ and $T_{A(p)}S^{2} = T_{-p}S^{2}$ are
canonically identified under this projection. Since $S^{2}$ cannot admit a globally defined Lorentzian metric, no Lorentzian
metric exists on $\mathbb{R}P^{2}$ either. Indeed, if a Lorentzian metric $g$
existed on $\mathbb{R}P^{2}$, then its pullback $\pi^{*}g$ would endow $S^{2}$
with a Lorentzian metric, which is impossible. However, using the quotient map $\pi$, one can equip $\mathbb{R}P^{2}$ with a
Riemannian metric by declaring
\[
\langle (d\pi)_{p}(v), (d\pi)_{p}(w) \rangle_{\pi(p)}
= \langle v, w \rangle_{p},
\qquad p \in S^{2},\ v,w \in T_{p}S^{2},
\]
which makes $\pi$ a local isometry.  
Alternatively, one may consider a more general Riemannian metric of the form
\[
ds^{2} = g(p)\, d\rho^{2},
\]
where $d\rho^{2}\!\mid_{p} = \langle \cdot, \cdot \rangle_{p}$ is the line element
of the unit sphere inherited from its embedding in Euclidean space, and $g(p) > 0$ for all
$p \in \mathbb{R}P^{2}$.
\end{rem}
 
For the reader's convenience, we reproduce Proposition~\ref{prop:crosscap-proposition} here.

\begin{prop*}[Radical and signature change on the crosscap]
Let $(C,g)$ be the 2-dimensional manifold (as defined in Equation~\ref{eq: adjunction space}) with metric  \[ g = (1-t^{2})\,dt^{2} + 2tx\,dt\,dx + (1-x^{2})\,dx^{2}, \] and define the degeneracy locus \[ \mathcal{H} = \{ (t,x) \in C : t^{2}+x^{2}=1 \}. \]
Then: \begin{enumerate} \item $(C,g)$ is a \emph{signature-changing manifold}, with Lorentzian signature in the region $t^{2}+x^{2} > 1$ and Riemannian signature in $t^{2}+x^{2} < 1$. \item The \emph{radical} at each point $q \in \mathcal{H}$ is 1-dimensional: \[ \mathrm{Rad}_{q} = \mathrm{span} \Bigl\{ (1, \sqrt{1-t^{2}}/t)^{T} \Bigr\} = \mathrm{span} \Bigl\{ (\sqrt{1-x^{2}}/x,1)^{T} \Bigr\}, \] where $T$ denotes the transpose of the matrix with respect to the chosen coordinates. \item The radical is of mixed character with respect to $\mathcal{H}$ ($\mathrm{Rad}_{q}$ is tangent iff $t^2 = x^2 = \frac{1}{2}$), with the radical being tangent at a measure-zero set. \end{enumerate}
Consequently, $(C,g)$ provides a compact example of a signature-changing manifold that \emph{cannot} be obtained from a Lorentzian manifold via the Transformation Prescription. 
\end{prop*}

\begin{proof}
On the complement $\mathbb{R}P^{2}\setminus\mathbb{\overline{D}}^{2}$,
we now introduce a Lorentzian-type metric that undergoes a transition
of signature along a compact hypersurface. With respect to the canonical coordinates $(t,x)\in((-\sqrt{2},\sqrt{2})\times[-\sqrt{2},\sqrt{2}])\setminus\mathbb{\overline{D}}^{2}$,
we define the Lorentzian $2$-manifold by the smooth metric

\[
g=(1-t^{2})(dt)^{2}+2txdtdx+(1-x^{2})(dx)^{2},
\]
where the quotient map identifies 
\[
(t,-\sqrt{2})\sim(-t,\sqrt{2}),\qquad(-\sqrt{2},x)\sim(\sqrt{2},-x).
\]

\textbf{~}\\ The determinant of the metric is 
\[
\det([g_{\mu\nu}])=(1-x^{2})(1-t^{2})-(tx)^{2}=1-t^{2}-x^{2}.
\]
Hence, $\det([g_{\mu\nu}])<0\Longleftrightarrow t^{2}+x^{2}>1$, which is non-vanishing
for $(t,x)\in((-\sqrt{2},\sqrt{2})\times[-\sqrt{2},\sqrt{2}])\setminus\mathbb{\overline{D}}^{2}$.
Therefore, $(\mathbb{R}P^{2}\setminus\mathbb{\overline{D}}^{2},g)$
is Lorentzian. Notice that the metric reduces to the
canonical form $g=-(dt)^{2}+(dx)^{2}$ when $x=0$ and $t=\pm\sqrt{2}$.\footnote{Choosing $x=\pm\sqrt{2}$ and $t=0$ has the effect of interchanging
the role of space and of time, and we get $g=(dt)^{2}-(dx)^{2}$.}

\textbf{~}\\ \textit{Relation with the crosscap:} The topological space $C$ defined
as the adjunction space $C = \overline{\mathbb{D}}^2 \cup_\psi \mathbb{M}$
(cf. Equation~\eqref{eq: adjunction space}) is topologically equivalent to the real projective plane $\mathbb{R}P^2$, obtained by sewing a
M\"obius strip to the boundary of a disk. The punctured crosscap
is homeomorphic to a M\"obius strip (up to the usual self-intersections
of its immersion in $\mathbb{R}^3$). Conversely, sewing a M\"obius strip
to a disk along their common boundary, using the attaching map $\psi$,
recovers the original topological crosscap $C$. We then define the crosscap as the adjunction space
\[
C \coloneqq \overline{\mathbb{D}}^2 \cup_h \overline{\mathbb{M}}
= \overline{\mathbb{D}}^2 \cup_h (\mathbb{M} \cup \partial \mathbb{M}),
\]
where $h$ is a diffeomorphism between collar neighborhoods of 
$\partial \overline{\mathbb{D}}^2$ and $\partial \mathbb{M}$. 
This extends the original attaching map $\psi$ from the primary definition. We then extend the metric $g$  across the boundary 
\[
\partial\mathbb{M}=\{(-\sqrt{2},x)\colon x\in[-\sqrt{2},\sqrt{2}]\}\cup\{(\sqrt{2},x)\colon x\in[-\sqrt{2},\sqrt{2}]\},
\]
so that the adjunction space $\mathbb{\overline{D}}^{2}\cup_{h}\overline{\mathbb{M}}$
becomes a semi-Riemannian manifold with underlying topological structure
$\mathbb{R}P^{2}$.

\textbf{~}\\ \textit{Metric on the crosscap:} We start with the square-shaped
domain $[-\sqrt{2},\sqrt{2}]\times[-\sqrt{2},\sqrt{2}]$ and impose
the antiparallel edge identifications

~

$(t,-\sqrt{2})\sim(-t,\sqrt{2})$ for $-\sqrt{2}\leq t\leq\sqrt{2}$,

$(-\sqrt{2},x)\sim(\sqrt{2},-x)$ for $-\sqrt{2}\leq x\leq\sqrt{2}$. 

\textbf{~}\\ This construction yields the non-orientable, compact
projective plane $\mathbb{R}P^{2}$, without boundary. Through continuous
deformation, $\mathbb{R}P^{2}$ can be transformed into the crosscap
\[
C\coloneqq\mathbb{\overline{D}}^{2}\cup_{h}\overline{\mathbb{M}}=\mathbb{\overline{D}}^{2}\cup_{h}(\mathbb{M}\cup\partial\mathbb{M}),
\]
with $\partial\mathbb{M}=\{(t,x)\colon t^{2}+x^{2}=1\}\cong S^{1}$.
The resulting $2$-dimensional signature-type changing manifold
$(C,g)$ in the $(t,x)$-plane carries the metric 

\[
g=(1-t^{2})(dt)^{2}+2txdtdx+(1-x^{2})(dx)^{2}.
\]

\textbf{~}\\ The determinant degenerates precisely along the hypersurface
$\partial\mathbb{M}\cong\mathbb{S}^{1}\subset C$: 
\[
\det([g_{\mu\nu}])=(1-x^{2})(1-t^{2})-(tx)^{2}=0\Longleftrightarrow t^{2}+x^{2}=1,
\]
so that 
\[
\mathcal{H}\coloneqq\{(t,x)\in C\colon t^{2}+x^{2}=1\}\subset C
\]
is the compact locus of signature change. Thus, 
\[
C\setminus\mathcal{H=}\mathbb{D}^{2}\cup\mathbb{M},
\]
where $\mathbb{D}^{2}$ has Riemannian signature and $\mathbb{M}$
has Lorentzian signature (see Figure~\ref{fig:Lightcones+Curves}).

\textbf{~}\\ Finally, the differential of the determinant is
\[
d(\det([g_{\mu\nu}]))=-2tdt-2xdx\begin{cases}
\begin{array}{c}
=0\\
\neq0
\end{array} & \begin{array}{c}
if\,x=t=0,\\
otherwise,
\end{array}\end{cases}
\]
which vanishes only at the origin $(t,x)=(0,0)$ and is non-zero everywhere
else on $\mathcal{H}$. 

Furthermore, the calculation of the radical $\textrm{Rad}_{q}$, for
$q\in\mathcal{H}$, reveals that $\dim(\textrm{Rad}_{q})=1$, and that
it is both transverse and tangent, depending on the point $q$:

\textbf{~}\\ We seek all vectors $w=(w_{1},w_{2})^{T}\in T_{q}M$
such that $g(w,\centerdot)=0$, where $T$ denotes the transpose of the matrix with respect to the chosen coordinates. Since the radical consists precisely of the null vectors of $g$ along
$\mathcal{H}$, a vector $w=(w_{1},w_{2})^{T}$ lies in $\mathrm{Rad}_{(t,x)}$ if and
only if
\[
(w_{1}, w_{2})
\begin{pmatrix}
1 - t^{2} & t x \\
t x       & 1 - x^{2}
\end{pmatrix}
\begin{pmatrix}
w_{1} \\[0.3em] w_{2}
\end{pmatrix}
= 0,
\]
which is equivalent to
\[
(1 - t^{2}) w_{1}^{2} + 2 t x\, w_{1}w_{2} + (1 - x^{2}) w_{2}^{2} = 0
\quad\Longleftrightarrow\quad
w_{1}^{2} + w_{2}^{2} + \frac{2 t x}{1 - t^{2}}\, w_{1} w_{2} = 0.
\]

Because points on the hypersurface $\mathcal{H}$ satisfy $t^{2} + x^{2} = 1$, we obtain
\[
(x w_{1} + t w_{2})^{2} = 0.
\]
Thus the null condition reduces to the single linear relation
\[
x w_{1} + t w_{2} = 0.
\]

On the other hand, differentiating the defining equation $t^{2} + x^{2} = 1$
for $\mathcal{H}$ yields the tangency condition
\[
t\, dt + x\, dx = 0.
\]
Replacing $dt$ and $dx$ by $\tfrac{dt}{d\lambda}$ and $\tfrac{dx}{d\lambda}$,
respectively (for an arbitrary curve parameter $\lambda$), we may regard them as
components of a tangent vector just like $w_{1}$ and $w_{2}$.

If the radical were tangent to $\mathcal{H}$, then both
\(
\begin{pmatrix}
w_{1}\\ w_{2}
\end{pmatrix}
\)
and
\(
\begin{pmatrix}
dt\\ dx
\end{pmatrix}
\)
would represent tangential directions, and together with the null condition
$x w_{1} + t w_{2} = 0$ we would obtain
\[
\begin{pmatrix}
x & t \\
t & x
\end{pmatrix}
\begin{pmatrix}
w_{1} \\ w_{2}
\end{pmatrix}
= 0.
\]
Nontrivial solutions exist precisely when
\[
\det
\begin{pmatrix}
x & t \\
t & x
\end{pmatrix}
= 0,
\qquad\text{i.e.}\qquad |t| = |x|.
\]
Together with $t^{2} + x^{2} = 1$, this yields $|t| = |x| = 1/\sqrt{2}$.
For all other points on $\mathcal{H}$, the radical is transverse, in agreement
with Fig.~\ref{fig:Lightcones+Curves}.
\end{proof}

The manifold $(C,g)$ is compact and neither pseudo-time orientable nor orientable. Time begins
at the surface of transition $\mathcal{H}$, defined by $t^{2}+x^{2}=1$,
but it is clear that spacetime does not originate there. Pseudo-timelike
curves can cross $\mathcal{H}$: they enter the Riemannian region
(the disk), traverse it, and re-emerge in the Lorentzian regime through
another point on $\mathcal{H}$. For example (see Figure~\ref{fig:Lightcones+Curves})
a pseudo-timelike curve $\gamma_2$ passes through point $p_2\in\mathcal{H}$,
traverses the disk in the Riemannian regime, and then re-emerges into
the Lorentzian regime through a point $q_2\in\mathcal{H}$. This curve
is future-directed. In contrast, the curve $\gamma_1$ enters the Riemannian regime at point $p_1\in\mathcal{H}$ and experiences
a time reversal upon re-entering the Lorentzian regime at point $q_1\in\mathcal{H}$.
Thus, this curve is not future-directed.

\textbf{~}\\ Finally, note that the metric 
\[
g=(1-t^{2})(dt)^{2}+2txdtdx+(1-x^{2})(dx)^{2}
\]
is not a transverse, type-changing metric with a purely transverse
radical, since the radical is of mixed character.
Consequently, we cannot derive this metric from a Lorentzian background
metric by means of a transformation function $f$. The Transformation
Prescription therefore does not apply in this case.

\begin{figure}[H]
\centering{}\includegraphics[scale=1.11]{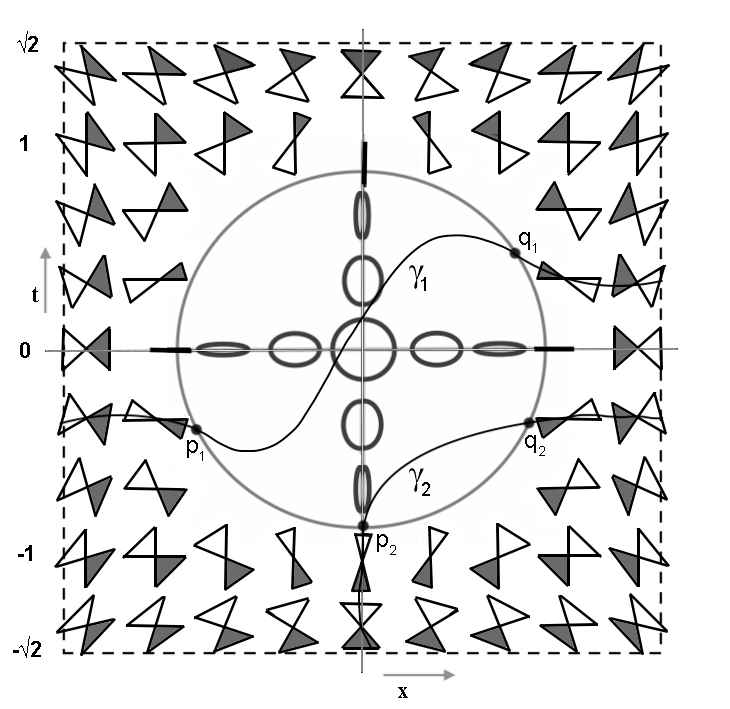}\caption{{\small{}\label{fig:Lightcones+Curves}Schematic light cone structure (not drawn to scale) for the metric $g=(1-t^{2})(dt)^{2}+2txdtdx+(1-x^{2})(dx)^{2}$. Pseudo-timelike
curves pass through $\mathcal{H}$, then go through the disk which
is the Riemannian regime, and finally re-emerge in the Lorentzian
regime through $\mathcal{H}$ again.}}
\end{figure}

\begin{cor}[Mixed character of the radical]
For the metric
\[
g = (1-t^{2})\, dt^{2} + 2tx\, dt\, dx + (1-x^{2})\, dx^{2},
\]
the radical along the hypersurface 
\[
\mathcal{H} = \{(t,x) \in \mathbb{R}^{2} : t^{2} + x^{2} = 1\}
\]
is tangent to $\mathcal{H}$ at exactly four isolated points, namely those satisfying 
\[
|t| = |x| = \frac{1}{\sqrt{2}},
\]
and is transverse to $\mathcal{H}$ at all other points.  
In particular, the radical has mixed character along $\mathcal{H}$.
\end{cor}

\begin{example}
The $2$-dimensional analogues of the ``no boundary''
proposal spacetimes, obtained by cutting a sphere
along its equator and attaching it to half of de Sitter space, exhibit
a radical that is always transverse along the signature-changing locus,
see Figure~{\ref{fig: 2D No Boundary Model}. Such
manifolds also cannot be constructed via the Transformation Prescription
from a Lorentzian manifold $(M,g)$. For further discussion, see~\cite{Hasse + Rieger-Transformation}.}

\begin{figure}[H]
\centering{}\includegraphics[scale=0.66]{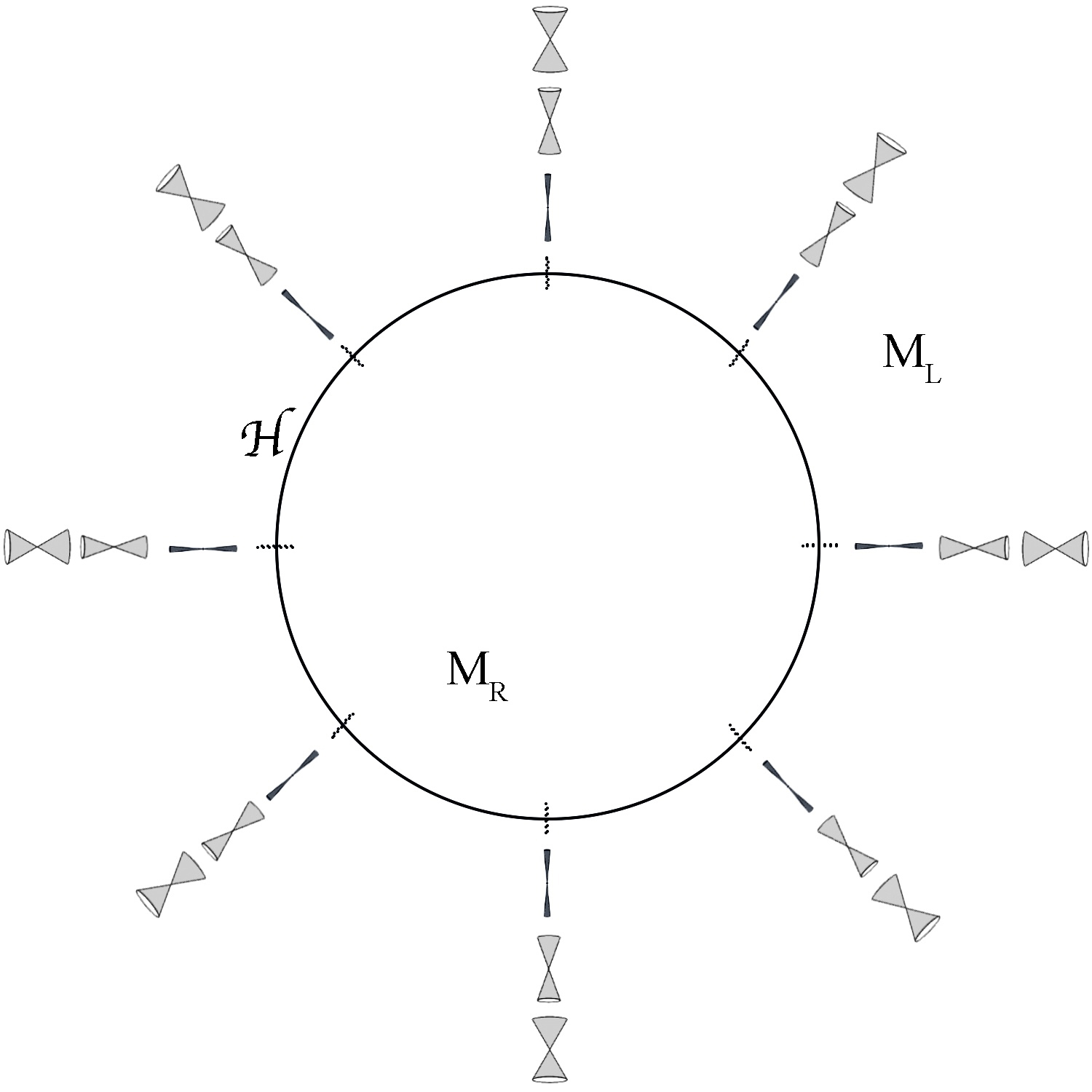}\caption{{\small{}\label{fig: 2D No Boundary Model}The causal structure of
the $2$-dimensional ``no boundary'' proposal spacetime.}}
\end{figure}
\end{example}

\vspace{0.001\baselineskip}

\begin{example}\label{Example-8}
Consider the $2$-dimensional signature-type-changing crosscap manifold
\[
\overline{\mathbb{D}}^{2} \cup_{h} (M \cup \partial M),
\]
constructed by attaching a M\"obius strip $M$ to the boundary of a disk
representing the Riemannian sector. This manifold is equipped with the metric
\[
g = (1-t^{2})\,(dt)^{2} + 2 t x\, dt\, dx + (1-x^{2})\, (dx)^{2},
\]
see Subsection~\ref{subsec:Cross-Cap-metric}.
According to Definition~\ref{def:Pseudo-space-orientable}, in order for the
crosscap to be pseudo-space orientable, we must exhibit a set of
$2-1 = 1$ pointwise orthonormal spacelike vector fields on the Lorentzian
region of
\[
\overline{\mathbb{D}}^{2} \cup_{h} (M \cup \partial M).
\]
The crosscap is closed and has Euler characteristic $\chi = 1$.\footnote{Here,
``closed'' is meant in the manifold sense of ``a compact manifold without
boundary'', not in the topological sense of ``the complement of an open
subset of $\mathbb{R}^{n}$''.}
Since a closed, connected manifold admits a non-vanishing vector field
if and only if its Euler characteristic is zero, there is no global
non-vanishing vector field on the crosscap. Nevertheless, the M\"obius strip
$M \cup \partial M$ admits a continuous non-vanishing spacelike
$(n-1)$-frame field. Hence, the manifold analyzed in
Subsection~\ref{subsec:Cross-Cap-metric} is indeed pseudo-space orientable
as well as pseudo-time orientable.
\end{example}

\section{Conclusion}

In this work, we constructed and analyzed explicit models of singular
semi-Riemannian manifolds undergoing signature-type change, with particular
emphasis on non-orientable topologies modeled on the M\"obius strip and the
crosscap. The examples show that non-orientable manifolds may support
signature-changing metrics, but that their global topology imposes substantial
restrictions on the applicability of standard transformation prescriptions. The analysis of our models yielded three main results.

~

First, in the rotating Lorentzian background, the stationary stripes exhibit a
one-way causal-barrier structure. With respect to the global time orientation
determined by
\[
V=(\cos\varphi)\partial_t-(\sin\varphi)\partial_x,
\qquad \varphi=\pi x,
\]
future-directed causal curves which enter an even stationary stripe \(M_{2k}\)
cannot leave it, whereas future-directed causal curves cannot enter an odd
stationary stripe \(M_{2k-1}\). Thus the causal trapping is a feature of the
Lorentzian background geometry and its time orientation, rather than a phenomenon
occurring at the hypersurface of signature change.

~

Second, the M\"obius-strip constructions show that pseudo-orientability is weaker
than ordinary orientability. In particular, a transverse signature-type changing
manifold may be both pseudo-time orientable and pseudo-space orientable while
failing to be orientable as a smooth manifold. This illustrates that the
pseudo-orientability notions relevant to signature-changing geometry do not
recover ordinary orientability.

~

Third, the compact crosscap model exhibits a fundamental obstruction to the
Transformation Prescription. At the hypersurface of signature change, the radical
of the induced metric has mixed character: it is transverse at some points and
tangent at others. Hence the metric does not satisfy the transverse-radical
condition required by the Transformation Theorem. Consequently, the crosscap
provides a compact non-orientable example in which the standard ansatz
\[
\tilde g = g + f(V^\flat\otimes V^\flat)
\]
cannot produce a signature-type changing metric with an everywhere transverse
radical.

~ 

Taken together, these examples show that signature change is not governed solely
by local coordinate models or by the choice of interpolation function \(f\).
Rather, global topological features, including non-orientability and Euler
characteristic, impose intrinsic restrictions on the existence of transverse
type-changing metrics. This suggests that any extension of the Transformation
Prescription to compact non-orientable manifolds must allow for more general
radical behavior, including mixed or tangent radicals.

\textbf{~}\\  Possible directions for future work include developing transformation procedures
adapted to non-transverse radicals, clarifying the role of characteristic classes
in obstructing transverse type change, and studying how boundary or junction
conditions should be formulated when the radical changes character along the
signature-change locus.

~
\begin{acknowledgement*}
The author acknowledges partial support of the SNF Grant No. 200020-192080.
This research was (partly) supported by the NCCR SwissMAP, funded
by the Swiss National Science Foundation. The author sincerely thanks Wolfgang Hasse for his collegiality, constant support, numerous discussions, and incisive critique, all of which greatly contributed to this work.
\end{acknowledgement*}

\section*{Declarations}

\noindent \text{Conflict of Interest:} The authors declare that they have no conflicts of interest.

\smallskip

\noindent \text{Ethical Statement:} This article is a purely theoretical work; no ethical approval was required for the research described.

\smallskip

\noindent \text{Data Availability:} Data sharing is not applicable to this article as no datasets were generated or analyzed during the current study.

\end{document}